\documentclass{amsart}
\usepackage{fullpage}
\usepackage[english]{babel}
\usepackage{amsmath}
\usepackage{amsthm}
\usepackage{amsfonts}
\usepackage{amssymb}
\usepackage{xypic}
\usepackage{longtable}
\usepackage{amsthm}
\usepackage{booktabs}
\usepackage{caption}
\usepackage{color}

\usepackage{amsmath}
\usepackage{graphicx}
\usepackage[colorlinks=true, allcolors=blue]{hyperref}

\newtheorem{thm}{Theorem}[section]
\newtheorem*{thx}{Theorem}

\newtheorem{lem}{Lemma}[section]
\newtheorem{rem}{Remark}[section]
\newtheorem{que}{Question}
\newtheorem{prop}{Proposition}[section]
\newtheorem{defi}{Definition}[section]

\newcommand{\abrack}[1]{[\mkern-3mu[#1]\mkern-3mu]}
\begin{document}
\title{Complex quaternionic manifolds and c-projective structures}
\author{Aleksandra Borówka}
\address[1]{Institute of Mathematics, Jagiellonian University in Krak\'ow\\Poland}
\email{aleksandra.borowka@uj.edu.pl}
\subjclass[2020]{53C26,53C28}
\keywords{Complex Quaternionic manifold, Quaternionic Feix--Kaledin construction, holonomy}

\begin{abstract}
We discuss complex quaternionic manifolds, i.e., those that have holonomy $GL(n,\mathbb{H})U(1)$, which naturally arise via quaternionic Feix--Kaledin construction. We show that for a fixed c-projective class, any real analytic connection with type $(1,1)$ curvature induces, via quaternionic Feix--Kaledin construction, an $S^1$-invariant connection with holonomy contained in  $GL(n,\mathbb{H})U(1)$. As an application, we characterize in this setting the distinguished $U^*(2n):=SL(n,\mathbb{H})U(1)$ connection studied in Battaglia \cite{Bat} and Hitchin \cite{Hit3}.
\end{abstract}
\maketitle
\section{Introduction}
A quaternionic manifold is a smooth manifold of dimension $4n$, $n>1$ such that it is equipped with a rank $3$
subbbundle $Q$ of the endomorphisms of $TM$ generated locally by three almost complex structures satisfying quaternionic relations and such that there exists a torsion-free connection $D$ with $DQ\subset Q$. Such connections, called quaternionic connections, form an affine space modeled on $1$-forms.

A natural question that one could ask is whether there always exists an integrable complex structure that is a section of $Q$. Globally this may not be the case as e.g., on $\mathbb{HP}^n$ for topological reasons there is no global almost complex structure (but see Pontecorvo \cite{Pon} for compact examples of quaternionic manifolds globally admitting such structure), but as shown by Alekseevsky et al. \cite{AMP}, locally such complex structures do exist.

On the other hand, for a fixed almost complex structure $I$, one could ask when there exists a quaternionic connection with $DI=0$. The holonomy group of such a connection is contained in $GL(n,\mathbb{H})U(1)$ and we call manifold equipped with such a connection a complex quaternionic manifold.  It is important to note that there is some ambiguity in the notation: some authors (Joyce, \cite{Joyce}, Hitchin \cite{Hit3}) define a  complex quaternionic  manifold as a manifold with holonomy $U^*(2n):=SL(n,\mathbb{H})U(1)$. This is equivalent to require that a quaternionic connection with $DI=0$ preserves also some volume form.

In Alekseevsky et al. \cite{AMP} it is shown that  a torsion free connection preserving a fixed almost complex structure is necessarily unique and exists exactly when the almost complex structure $I$ is integrable. However, it seems that this result has been somewhat overlooked (see for example, Thurman \cite{T}, Section 4.1).  The main aim of this note is to describe an interesting class of examples of $S^1$- invariant complex quaternionic structures that naturally arise from quaternionic Feix--Kaledin construction \cite{BC}. More precisely, we show
\begin{thx}
    Let $(S,J,[D])$ be a real analytic c-projective manifold with c-projective curvature of type $(1,1)$. Let $(\mathfrak{L},\nabla)$ be a holomorphic line bundle with a holomorphic connection with real analytic type $(1,1)$ curvature. Then locally near $S$, the quaternionic manifold $M$ obtained by quaternionic Feix--Kaledin construction from $(S,J,[D]_c,\mathfrak{L},\nabla) $  admits a family $J_D$ of compatible $S^1$-invariant complex structures parametrized by real analytic connections $D$ in the c-projective class with curvature of type $(1,1)$ and such that $J_D|_S=J$. 
\end{thx}

As a corollary, we deduce that quaternionic manifolds near maximal fixed point sets of a rotating quaternionic $S^1$-action admit the family of $S^1$-invariant connections with holonomy contained in $GL(n,\mathbb{H})U(1)$ are parametrized by real analytic connections $D$ with curvature of type $(1,1)$ in the underlying c-projective class .

A particular case of our examples of complex structures coincide with the distinguished complex structure on quaternion-K\"ahler manifolds with isometic circle action found by Battaglia \cite{Bat}. Therefore our results provide a characterization of this complex structure in the case when the $S^1$-action admits a fixed point set of maximal dimension.
It is worth noting that Hitchin \cite{Hit3} shows that the holonomy of the underlying connection is in fact  $U^*(2n)$. 
\subsection*{Acknowledgements}
The author has been supported by the Polish National Science Center project number 2022/47/D/ST1/02197. 
\section{Background}
 In this section we give necessary definitions and a very brief description of quaternionic Feix--Kaledin construction \cite{BC}. For more detailed background, see \cite{Alex},  \cite{Bo}, \cite{BC} and \cite{SAL}. 

\begin{defi}
Let $n\geq 2$. A $4n$-dimensional smooth manifold $M$ is called quaternionic if it is equipped  with a subbundle $Q\subset End (TM)$ with $rk(Q)=3$ such that $Q$ is locally generated by three anti-commuting almost complex structures $I,J,K$ satisfying 
$$I^2=J^2=K^2=IJK=-Id,$$
and such that there exists a torsion-free connection $D$, called a quaternionic connection, with $D_XQ\subset Q$ for all $X$. 
\end{defi}

For a fixed quaternionic manifold the class of all quaternionic connections $[D]_q$ is  given by an equivalence relation

\begin{equation}\label{q-structure} D\sim_q D' \Leftrightarrow \exists_{\gamma}:\ \forall Y,Z\in TM \quad D'_YZ=D_YZ +\abrack{Y,\gamma}_qZ,\end{equation}
where $\gamma\in T^*M$ and $$\abrack{Y,\gamma}_q(Z)=\frac{1}{2}(\gamma (Y)Z+\gamma (Z)Y-\Sigma_{i=1}^3(\gamma(I_iY)I_iZ+\gamma(I_iZ)I_iY))$$ where $I_1,I_2,I_3$ is a local frame of almost complex structures (see \cite{Alex}).

A quaternionic manifold admits many $Q$-hermitian metrics. However, usually the Levi-Civita connection of such metric is not quaternionic. In a rare case when it is quaternionic we call the metric \emph{quaternion-K\"ahler} and in fact, it is one of the special holonomy cases from the Berger list \cite{Be}, where the group is $Sp(n)\cdot Sp(1) $. 

Another special case of quaternionic manifold is when the bundle $Q$ is trivial and generated by three complex structures $I,J, K$ satisfying quaternionic relation. In this case the manifold is called \emph{hypercomplex} and it admits unique quaternionic connection with $DI=DJ=DK=0$ called the \emph{Obata connection}. In the case when $D$ is the Levi-Civita connection of a metric which is K\"ahler with respect to all $(I,J,K)$ we call the manifold \emph{hyperk\"ahler}.

The main object of this paper is \emph{complex quaternionic structure}, i.e., a quaternionic manifold with a fixed compatible integrable complex structure. Now, we will recall some facts about them due to D. Alekseevsky, S. Marchiafava and M. Pontecorvo. 
\begin{prop}[Alekseevsky et al. \cite{AMP}]
 Let $(M,Q)$ be a quaternionic manifold of dimension $4n>4$. Then locally on $M$ there exists an integrable complex structure.   
\end{prop}

\begin{thm}[Alekseevsky at al. \cite{AMP}, Proposition 3.6 and Theorem 4.1] Let $I$ be an integrable complex structure on  some quaternionic manifold $(M,Q)$. Then, there exists a unique quaternionic connection $D$ preserving $I$, i.e., satisfying $DI=0$.
    
\end{thm}

Quaternionic connections preserving some integrable complex structure have holonomy contained in $GL(n,\mathbb{H})U(1)$.
\begin{rem}

Note that there is some ambiguity in terminology as Joyce \cite{Joyce} and Hitchin \cite{Hit3} by  quaternionic complex connections mean quaternionic manifolds which preserves both an integrable complex structure  and some volume form. We call such connection special complex quaternionic connections (see \cite{BC}) (note that in \cite{T} they are called scaled quaternionic). Special complex quaternionic connections have holonomy inside group $SL(n,\mathbb{H})U(1)$ which in \cite{Hit3} is also denoted by $U^*(2n)$ to emphasize its analogy with $U(2n)$ as a real form of $GL(2n,\mathbb{C})$.
\end{rem}

For any quaternionic manifold $(M,Q)$  over each point of $M$ we have a sphere (i.e., $\mathbb{CP}^1$) of anti-commuting complex structures, i.e. those elements of $Q$ that square to $-Id$. The \emph{twistor space} $Z$ of $M$ is defined to be the total space of this sphere bundle. Salamon \cite{SAL} showed that $Z$ admits an integrable complex structure and that the family of projective lines given by fibres of the sphere bundle over $M$ have normal bundle isomorphic to $\mathcal{O}(1)\otimes\mathbb{C}^{2n}$.

Moreover, the following theorem is true:
\begin{thm}[Pedersen--Poon, \cite{PP}] \label{TwSal}
Let $Z$ be a complex manifold of dimension $2n+1$ such that:
\begin{itemize}
\item[(i)]There is a family of holomorphic projective lines $\mathbb{C}\mathbb{P}^1$ each with normal bundle isomorphic to $\mathbb{C}^{2n}\otimes\mathcal{O}(1)$,
\item[(ii)]Z has a real structure which on lines from the family which are invariant under this real structure induces the antipodal map of $\mathbb{C}\mathbb{P}^1$.
\end{itemize}
Then the parameter space of projective lines invariant under the real structure, called real twistor lines, is a $4n$-manifold with natural quaternionic structure for which $Z$ is the twistor space.
\end{thm}
Many additional structures and properties of quaternionic structures can be encoded in the underlying twistor spaces. In particular, the existence of quaternion-K\"ahler metric corresponds to the existence of a holomorphic contact distribution on $Z$ and the existence of a quaternionic $S^1$-action on $M$ corresponds to a local holomorphic $\mathbb{C}^*$-action on $Z$. Both the distribution and the action need to be compatible with the real structure and  transversal to real twistor lines.
 \begin{defi}
 Suppose that $x\in M$ is a fixed point set of a quaternionic $S^1$-action. Consider the $\mathbb{C}^*$-action on the twistor line corresponding to $x$ in $Z$.
 \begin{itemize}
     \item If the action on the twistor line is trivial then we call $x$ a tri-holomorphic point,
     \item If the action rotates the twistor line then we call $x$ a rotating point.
 \end{itemize} 
 \end{defi}
 Note that any rotation of $\mathbb{CP}^1$ fixes exactly two antipodal points.

Recall that a \emph{projective structure} on a smooth manifold is an equivalence class of connections that have the same geodesics. It can be also considered in a holomorphic setting, where we consider only holomorphic connections (i.e. connections on complex manifolds with holomorphic Chistoffel symbols). If such a holomorphic projective structure on a complex manifold $S$ is flat (which for complex dimension $n>1$ means that the invariant part of the curvature - called projective Weyl curvature - vanishes) then there exists unique embedding of $S$ into the model $\mathbb{CP}^n$ such that the projective structures are compatible. 
\begin{defi}
A $2n$-manifold $(M,J,[D]_c)$, $n>1$, is called a \emph{c-projective} manifold if $J$ is an integrable complex structure and $[D]_c$ is a c-projective structure, i.e., an equivalence class of complex $(DJ=0)$ torsion-free connections under the equivalence

$$D\sim_c D' \Leftrightarrow \exists_{\gamma}:\ \forall Y,Z\in TM \quad D'_YZ=D_YZ +\abrack{Y,\gamma}_cZ,$$
where $\gamma\in T^*M$ and $\abrack{Y,\gamma}_c(Z)=\frac{1}{2}(\gamma (Y)Z+\gamma (Z)Y-\gamma(JY)JZ-\gamma(JZ)JY)$.

\end{defi}
Note that in older literature, the c-projective structure was called h-projective structure.

 Projective, c-projective and quaternionic structures are examples of projective parabolic Cartan geometries. In such geometries an important role is played by density bundles. In case of  projective structures in holomorphic setting, this is the holomorphic line bundle $\mathcal{O}(1)$ over a projective manifold $S$, where \begin{equation}
     \mathcal{O}(1)^{\otimes (n+1)}=\bigwedge^{n}TS
 \end{equation}\label{c}

 Note that in general the bundle $\mathcal{O}(1)$ is defined only locally on $S$, as roots of bundles may exist only locally. However, when $S$ is equal to $\mathbb{CP}^n$ with the standard projective structure, the bundle $\mathcal{O}(1)$ is defined globally and is the dual of the tautological line bundle.
 
\begin{equation}\label{tra}
    \mathcal{D}^D_Y\left(\begin{array}{c}l\\\alpha\end{array}\right)=\left(\begin{array}{c}D_Yl-\alpha(Y)\\ D_Y\alpha+(r_p^D)_Yl\end{array}\right),
\end{equation}
where  $r_p^D$ is some normalization of the Ricci tensor.
One can show that this connection when considered on $J^1\mathcal{O}(1)$ does not depend on the choice of the connection in the class and that $\mathcal{D} $ is flat iff the projective structure is flat. 
\subsection{Quaternionic Feix--Kaledin construction} Now we will give a very brief description of quaternionic Feix--Kaledin construction \cite{BC}. As the construction is quite complicated, we limit only to the facts that are needed to construct examples in Section \ref{fami}. 

The idea behind quaternionic Feix--Kaledin is to fully characterize quaternionic structures on a neighbourhood of a component $S$ of the fixed point set of a quaternionic $S^1$-action provided that the action is rotating and $S$ is of maximal dimension (which is half of the dimension of the quaternionic manifold) in terms of a structure on $S$. It turns out that such $S$ is a complex submanifold and is equipped with a c-projective structure with c-projective Weyl curvature (i.e. invariant part of the curvature) of type $(1,1)$ together with a holomorphic line bundle with a holomorphic connection with curvature of type $(1,1)$. Moreover, these data fully characterize (locally near $S$) the underlying quaternionic manifold and therefore can be used to reconstruct it. 

\begin{thm}{\cite{BC}}
    Let $(S,J,[D]_c)$ be a real analytic c-projective manifold with c-projective curvature of type $(1,1)$. Let $(\mathfrak{L},\nabla)$ be a holomorphic line bundle with a holomorphic connection with real-analytic type $(1,1)$ curvature. Then the quaternionic Feix--Kaledin construction from \cite{BC} gives a twistor space $Z$ of a quaternionic manifold $M$ equipped with a quaternionic $S^1$-action having $S$ as a component of its fixed points.  Furthermore, $S$ is a totally complex submanifold of $M$, with induced c-projective structure $[D]_c$, and a neighbourhood of $S$ in $M$ is
$S^1$-equivariantly diffeomorphic to a neighbourhood of the zero section in
$TS\otimes(\mathcal{L}_{1,0}\otimes(\mathcal{L}_{0,1})^*)|_S$. Moreover, any quaternionic
$4n$-manifold with a quaternionic $S^1$-action whose fixed point set has a
connected component $S$ which is a submanifold of real dimension $2n$ with no
triholomorphic points arises from the induced c-projective structure on $S$ via the
quaternionic Feix--Kaledin construction, for some holomorphic line bundle $\mathfrak{L}$ on
$S$.
\end{thm}

The bundles $\mathcal{L}_{1,0}$ and $\mathcal{L}_{0,1}$ are some line bundles depending on $\mathfrak{L}$ which will be defined in Section \ref{p}.

In particular, the construction gives a precise description of the twistor space of quaternionic manifolds with $S^1$-action of the required type in a neighborhood of the twistor lines corresponding to the fixed point set $S$. To be able to construct the example, we need to understand the properties of these twistor spaces.

\subsubsection{Overview}
Let $Z$ be a twistor space obtained by the quaternionic Feix--Kaledin construction, i.e. it is a twistor space of a quaternionic manifold $M$ with a rotating $S^1$-action which preserves the quaternionic structure and has the fixed point set $S\subset M$  which is a complex manifold of complex dimension $n$. Then there exist two holomorphic vector bundles $V^+$ and $V^-$ of rank $n+1$ over $n$-dimensional complex manifolds and two open subsets $U^+\subset (V^+)^*$ and $U^-\subset (V^-)^*$ such that $Z$ is obtained as a gluing of $U^+$ and $U^-$. Moreover, the $\mathbb{C}^*$-action is induced on $Z$ from the scalar multiplication in the fibres of $(V^+)^*$ and the inverse of scalar multiplication in the fibres of $ (V^-)^*$. The general idea is, that the gluing is given by two blow-down maps to $U^+$ and $U^-$ from some bundle $\mathcal{P}$ of projective lines over neighborhood of the diagonal of the complex manifold $S\times \overline{S}$. The two zero sections of $U^+$ and $U^-$ correspond to the submanifold $S$ with two distinguished complex structures $J$ and $-J$. Below we discuss necessary technical details. 

\subsubsection{The bundle $\mathcal{P}$}\label{p} Observe that $S\times \overline{S}$ is a model of a complexification of the complex manifold $S$ and the diagonal is the real submanifold isomorphic with $S$ (for details concerning complexification see \cite{BC}). The projections on the first and second factor are holomorphic and define  on $S\times \overline{S}$ two transverse holomorphic foliations by manifolds isomorphic with $S$ and $\overline{S}$ respectively. 
\begin{defi}
    The foliation given by equation $S\times{b}$ for $b\in\overline{S}$ will be denoted the $(1,0)$-foliation and the foliation given by equation $a\times S$ for $a\in S$ will be denoted the $(0,1)$-foliation. Note that the leaf space of the $(1,0)$-foliation is isomorphic with $\overline{S}$ and the leaf space of the $(0,1)$-foliation is isomorphic with $S$.
\end{defi}\label{fol}
For the holomorphic bundles $\mathfrak{L}$ and $\mathcal{L}:=\mathfrak{L}\otimes \mathcal{O}(1)$ over $S$ (see \ref{c}) we consider their pull-backs  by the  projection on the first factor from $S\times\overline{S}$ and denote them by $\mathfrak{L}_{1,0}$ and $\mathcal{L}_{1,0}$ respectively. Note that these bundles are trivial along the leaves of the $(0,1)$-foliation. Similarly,  the bundles $\overline{\mathfrak{L}}$ and  $\overline{\mathcal{L}}$ are holomorphic over $\overline{S}$ (see \ref{c}) and their pull-back  by the  projection on the second factor are trivial along the leaves of $(1,0)$-foliation and will be denoted by $\mathfrak{L}_{0,1}$ and $\mathcal{L}_{0,1}$. The rank $2$ bundle $\mathcal{L}_{1,0}\oplus\mathcal{L}_{0,1}$ over $S\times\overline{S}$ is a complexification of the bundle $\mathcal{L}:=\mathfrak{L}\otimes\mathcal{O}(1)$ over $S$. We define the projective bundle over $S\times\overline{S}$ by $$\mathcal{P}:=\mathbb{P}(\mathcal{L}_{1,0}\oplus\mathcal{L}_{0,1})\ \ \ 
\texttt{and} \ \ \ \mathbb{P}(\mathcal{L}_{1,0}\oplus\mathcal{L}_{0,1})=\mathbb{P}([\mathcal{L}_{1,0}\otimes(\mathcal{L}_{0,1})^*]\oplus\mathcal{O})=\mathbb{P}(\mathcal{O}\oplus[(\mathcal{L}_{1,0})^*\otimes\mathcal{L}_{0,1}]).$$
These equalities show that the projective bundle admits two distinguished disjoint sections and can be viewed as a gluing of two vector bundles $\mathcal{L}_{1,0}\otimes(\mathcal{L}_{0,1})^*$ and $(\mathcal{L}_{1,0})^*\otimes\mathcal{L}_{0,1} $ by $l\mapsto l^*$ for which the two sections become zero sections and sections in the infinity.

\subsubsection{Flat structures along the leaves}\label{flat} On the other hand, one can show that the complexification of c-projective structure on a neighbourhood of the diagonal in $S\times\overline{S}$ gives flat holomorphic projective structures along each leaf of the two foliations defined in \ref{fol}. Moreover,  $\mathfrak{L}_{1,0}\oplus\mathfrak{L}_{0,1}$ over $S\times\overline{S}$ is a complexification of the bundle $\mathcal{L}$ over $S$ and one can prove (see \cite{BC}) that a complexification of $\nabla$ gives connection  on $\mathfrak{L}_{1,0}$ which is flat along leaves of the $(1,0)$-foliation and trivial along leaves of the the $(0,1)$-foliation and the analogous fact is true for the bundle $\mathfrak{L}_{0,1}$. Therefore, we have flat connections on $\mathfrak{L}_{1,0}\otimes(\mathfrak{L}_{0,1})^*$ and $(\mathfrak{L}_{1,0})^*\otimes\mathfrak{L}_{0,1} $ along the corresponding leaves.

These structures can be used to construct twisted tractor connections along each leaf, i.e., connections on $1$-jet bundles $J^1( \mathcal{L}_{1,0}\otimes(\mathcal{L}_{0,1})^*)$ and $J^1((\mathcal{L}_{1,0})^*\otimes\mathcal{L}_{0,1})$ along leaves of  the $(1,0)$-foliation and the $(0,1)$-foliation respectively defined as tensor product of tractor connections and the flat connections on $\mathfrak{L}_{1,0}\otimes(\mathfrak{L}_{0,1})^*$ and $(\mathfrak{L}_{1,0})^*\otimes\mathfrak{L}_{0,1} $.
\begin{rem}
    Note that as the projective tractor bundle is defined on $J^1(\mathcal{O}(1))$ and the bundle $ J^1(\mathcal{L}_{1,0}\otimes(\mathcal{L}_{0,1})^*)$ differs from the bundle $J^1(\mathfrak{L}_{1,0}\otimes(\mathfrak{L}_{0,1})^*) $ by a pull-back of $\overline{\mathcal{O}(1)}$ by projection on the second factor. However, this pull-back bundle is canonically trivial along the leaves of the $(1,0)$-foliation hence does not contribute into the connection. 
\end{rem}

\subsubsection{The bundles  $V^+$ and $V^-$}\label{bun} As over each leaf of the foliations \ref{fol} we have flat holomorphic connection on a vector bundle (of $1$-jets) which is of rank $n+1$, therefore we also have $(n+1)$-dimensional vector spaces of parallel sections of the connections. Using the evaluation of $1$-jets these  induce families of sections of $\mathcal{L}_{1,0}\otimes(\mathcal{L}_{0,1})^*$ and $(\mathcal{L}_{1,0})^*\otimes\mathcal{L}_{0,1} $  called \emph{affine} sections. We define two vector bundles over leaf spaces of the foliations by requirement that fibrewise they are the corresponding vector spaces of affine sections:
\begin{defi}\label{d:affine} 
The bundle $V^-\to S$ of \emph{affine
sections} along the fibres of the $(0,1)$-foliation \ref{fol}
is the bundle whose fibre at $x\in S$ is the space of sections $l$ of $(\mathcal{L}_{1,0})^*\otimes\mathcal{L}_{0,1}$ along the leaf such that there exists  a section $\hat{l}$ of $J^1[(\mathcal{L}_{1,0})^*\otimes\mathcal{L}_{0,1}]$ parallel with the respect to the twisted tractor connection on the leaf such that $ev(\hat{l})=l$, where $ev$ is the canonical projection from $J^1[(\mathcal{L}_{1,0})^*\otimes\mathcal{L}_{0,1}]$ to $(\mathcal{L}_{1,0})^*\otimes\mathcal{L}_{0,1}$.

The
bundle $V^+\to \overline{S}$ is defined similarly. 
\end{defi}
\subsubsection{The gluing map} We define two evaluation maps $\phi_+$ from $(\mathcal{L}_{1,0})^*\otimes\mathcal{L}_{0,1} $  to $(V^+)^*$ and $\phi_-$ from $\mathcal{L}_{1,0}\otimes(\mathcal{L}_{0,1})^*$ to $V^-$  by:
\begin{equation}
\phi_+: \ \ \ (x,\tilde{x},l) \mapsto (\tilde{x}, f\mapsto f(x)\otimes l)
\end{equation}
\begin{equation}
\phi_-: \ \ \ (x,\tilde{x},\tilde{l}) \mapsto (z, \tilde{f}\mapsto \tilde{f}(\tilde{x})\otimes \tilde{l}),
\end{equation}
where $(x,\tilde{x})$ denote points on the base $S\times\overline{S}$.
Note that the maps defined in this way behave like blow-down.

As explained in Section \ref{p}, the gluing of $(V^+)^*$ and $(V^-)^*$ is given by gluing of $\mathcal{L}_{1,0}\otimes(\mathcal{L}_{0,1})^*$ and $(\mathcal{L}_{1,0})^*\otimes\mathcal{L}_{0,1} $ induced from $\mathcal{P}$. 

A manifold obtained by such a gluing may not be Hausdorff, but we can choose open submanifolds $U^+\subset (V^+)^*$ and $U^-\subset (V^-)^*$ containing the zero sections such that the manifold $Z$  obtained by gluing is Hausdorff. In fact $Z$ is a twistor space of quaternionic manifold.
\begin{defi}
    We denote by $Z$
 the complex manifold obtained by gluing $U^+$ and $U^-$ via $\phi_-\circ *\circ(\phi_+)^{-1}$, where the map $*$ denotes dualization $l\mapsto l^*$ between dual bundles. We call $Z$ the twistor space of a quaternionic manifold obtained by quaternionic Feix--Kaledin construction from $(S,J,[D]_c,\mathfrak{L},\nabla)$.
 \end{defi}
\subsubsection{The $\mathbb{C}^*$-action} $Z$ is equipped with a holomorphic local $\mathbb{C}^*$-action given by a scalar multiplication in the fibres of $ (V^+)^*$ and  its inverse in $(V^-)^*$. The $S^1$ subgroup of the action preserves the real twistor lines and therefore induces a quaternionic action on $M$. The $\mathbb{C}^* $-action rotates real twistor lines joining the zero sections of $ (V^+)^*$ and $(V^-)^*$  --- they correspond to the submanifold $S\subset M$ --- therefore, the action is rotating on its fixed points set.

\section{$S^1$-invariant families of complex structures}\label{fami}
Now we are ready to construct the family 
of $\mathbb{C}^*$-invariant complex structures on $M$. These ideas were mentioned briefly in Remark 3 of \cite{Bo} but without giving details as the paper focuses only on a quaternion-K\"ahler case of the construction. 
Suppose that $D\in [D]$ is a real analytic connection in the c-projective class with curvature of type $(1,1)$. Then on $\mathcal{O}(1)$ (see Equation \ref{c}) the connection $D$ induces a connection with curvature of type $(1,1)$. Therefore, using the complexification of the connection $\nabla$ on $\mathfrak{L}$ as discussed in Section \ref{flat} we obtain two families of flat connections $\nabla^D_{w}$ on $\mathcal{L}_{1,0}\otimes(\mathcal{L}_{0,1})^*$ along the $(1,0)$-foliation and on $(\mathcal{L}_{1,0})^*\otimes\mathcal{L}_{0,1}$  along the leaves of the $(0,1)$-foliation respectively (here to shorten the notation, $w$ is a parameter of leaves in both families). The equation 
\begin{equation}\label{par}
\nabla^D_{w}l_w=0 
\end{equation}
define $1$-dimensional vector spaces of parallel sections $cl_w$ of  $\mathcal{L}_{1,0}\otimes(\mathcal{L}_{0,1})^*$ and $(\mathcal{L}_{1,0})^*\otimes\mathcal{L}_{0,1}$ along the leaves of the respective foliations, where $c\in \mathbb{C}$.
\begin{lem}
The families of sections defined by Equation \ref{par} are affine sections along the leaves defined in Section \ref{bun}. Therefore, this defines holomorphic line subbundles $L^+\subset V^+$ and $L^-\subset V^-$ (see Definition \ref{d:affine}).
\end{lem}
\begin{proof}
  We prove the statement for $L^+$, the proof for $L^-$ is analogous. We need to show that if $l_w$ satisfies Equation \ref{par}  then it extends to a section of the corresponding $1$-jet bundle parallel with respect to the twisted tractor connection along the respective leaf of the $(1,0)$-foliation. Using the connection $\nabla^D_w$ to induce the splitting of  $J^1(\mathcal{L}_{1,0}\otimes(\mathcal{L}_{0,1})^*)\cong [T^*S\otimes\mathcal{L}_{1,0}\otimes(\mathcal{L}_{0,1})^*]\oplus[\mathcal{L}_{1,0}\otimes(\mathcal{L}_{0,1})^*]$ we define the section of the $1$-jet bundle by $$\hat{l}:=(l_w,\nabla^D_wl_w)=(l_w,0).$$
  Using the formula for the projective tractor connection on $\mathcal{O}(1)$ (see Equation \ref{tra}) and the fact that the twisted tractor connection is a tensor product of the tractor connection with the flat connection on a line bundle which is compatible with $\nabla^D_w$ we deduce the claim.
\end{proof}
Define rank $n$ subbundles $D^+$ of $(V^+)^*$ and  $D^-$ of $(V^-)^*$ as annihilator bundles of $L^+$ and $L^-$: $$D^+_x=\{s\in (V^+_x)^*\ :\ s(L^+_x)=0\}\ \ \ \ \ \  D^-_{\tilde{x}}=\{\tilde{s}\in (V^-_{\tilde{x}})^*\ :\ \tilde{s}(L^-_{\tilde{x}})=0\}.$$  

\begin{prop}\label{main}
 The hypersurfaces $D^+\cap U^+\subset Z$ and $D^-\cap U^-\subset Z$ define up to a sign an integrable $S^1$-invariant complex structure on the quaternionic manifold $M$ obtained from the twistor space $Z$.
\end{prop}
\begin{proof}
  As any holomorphic hypersurface of the twistor space that is injective under the twistor projection defines a compatible integrable complex structure on an open subset of the quaternionic manifold $M$  (\cite{AMP}). Moreover, $D^+$ and $D^-$ are conjugated by the real structure on $Z$, hence they  up to a sign give the same complex structure (the real structure is induced by the antyholomorphic isomorphism between $V^+$ and $V^-$  induced by the canonical real structure on $S\times\overline{S}$ - see Proposition 4.3 in \cite{BC}). Therefore, it is enough to show that the $S^1$-action preserves the complex structure. This is equivalent to the statement that the corresponding hypersurface in the twistor space is preserved by the corresponding local $\mathbb{C}^*$-action. The $\mathbb{C}^*$-action is in our case induced by the scalar multiplication in the fibres of $V^+$ and $V^-$ and $D^+$ and $D^-$ are (open subsets) of vector subbundles hence they are preserved by the (local)   $\mathbb{C}^*$-action.
\end{proof}

\begin{rem}\label{uni}
The equations used to construct hyperspaces $D^+$ and $D^-$ from a connection $D$ in the c-projective class use only the induced connection on $\mathcal{O}(1)$. However, as visible from Remark 2.4 of \cite{CEMN}, the connections on the canonical bundle (and hence of its roots bundles) induced by different connections from the fixed c-projective class differ. Therefore, the complex structures obtained using different representatives in the c-projective class are different. 
\end{rem}
\begin{prop}\label{con}
    The $GL(n,\mathbb{H})U(1)$ connection corresponding to the complex structure obtained in  Proposition \ref{main} is invariant with respect to the $S^1$-action.
\end{prop}
\begin{proof}
  Let $\mathcal{D}^J$ be the $GL(n,\mathbb{H})U(1)$ connection with $\mathcal{D}^J J_D=0$. We know that for a generator $X$ of the action the push forward by the flow of $\mathcal{D}^J$ along $X$ by any fixed time  is some quaternionic connection $\mathcal{D}'$ (since $X$ preserves the quaternionic structure). Moreover, $\mathcal{D}'$ preserves the underlying push forward of $J$.  But as $L_XJ=0$, we get that  $\mathcal{D}'J=0$. As  $\mathcal{D}^J$ is the unique connection with  $\mathcal{D}^J J_D=0$, we get that  $\mathcal{D}^J =\mathcal{D}'$. 
\end{proof}
\begin{thm}
    Let $(M,Q)$ be a quaternionic $4n$-manifold with a rotating quaternionic circle action and $S$ a $2n$-dimensional component of the fixed points of the action. Suppose that $(M,Q)$ is not locally hypercomplex. Denote by $[D]_c$ the c-projective structure induced on $S$ by the quaternionic structure on $M$ and let $J$ be the induced complex structure on $S$. Then,  the set of real analytic $GL(n,\mathbb{H})U(1)$ connections on $M$ compatible with both $Q$ and  the circle action and preserving a complex structure that extends $J$ is parametrized by the set of real-analytic connections $D\in[D]_c$ with real analytic curvature of type $(1,1)$.
\end{thm}

\begin{proof}
By Proposition \ref{con} and Remark \ref{uni}, it is enough to show that any $GL(n,\mathbb{H})U(1)$ connection with required properties can be obtained using the construction from this section. Let $\mathcal{D}$ be such a connection. Then, there exists a complex structure $\mathcal{J}$ with $\mathcal{D}\mathcal{J}=0$. Moreover, such $\mathcal{J}$ is unique up to a sign unless $\mathcal{D}$ is an Obata connection of a hypercomplex structure. As $\mathcal{D}$ is $S^1$-invariant then so is $\mathcal{J}$. Hence $\mathcal{J}$ corresponds to a hypersurface in $Z$ (we abuse the notation and denote it also by $\mathcal{J}$), which is $\mathbb{C}^*$ invariant and transversal to real twistor lines. By assumption we know that $Z$ can be obtained by quaternionic Feix--Kaledin construction. Moreover $\mathcal{J}|_S=J$ hence the surface $\mathcal{J}$ contains the zero section of a vector bundle $(V^+)^*$. For any fibre $(V^+)^*_x$ of the bundle, as the $\mathbb{C}^*$-action preserves the fibres we have that  $(V^+)^*_x\cap\mathcal{J}$ is $\mathbb{C}^*$-invariant hence it is a vector subspace. Therefore, it defines a $1$-dimensional vector subspace of $V^+$ and hence $1$ dimensional subspace of affine sections of $\mathcal{L}_{1,0}\otimes(\mathcal{L}_{0,1})^*$ along the corresponding leaf. Similarly, as in the proof of Theorem 2 from \cite{Bo}, and taking into account the analogous argument for $-\mathcal{J}$ and $V^-$, we can use these data to obtain connections on $\mathcal{L}_{1,0}\otimes(\mathcal{L}_{0,1})^*$ and $(\mathcal{L}_{1,0})^*\otimes\mathcal{L}_{0,1}$ which can be obtained using the construction provided in this section.
\end{proof}
    \begin{thm}\label{wn}
Let $(M,g)$ be a quaternion-K\"ahler $4n$-manifold with rotating isometric circle action and suppose that $S$ is a $2n$-dimensional component of the fixed point set of the action.  Then, the distinguished $U^*(2n)$ connection $\mathcal{D}$ obtained by Battaglia \cite{Bat} and studied by Hitchin \cite{Hit3} is the unique $U^*(2n)$ connection which agrees with the Levi-Civita connection of $g$ on $S$.
    \end{thm}
\begin{proof}
   Comparing the complex structure from Lemma 1 of \cite{Bo} with the complex structure that we obtain in  Proposition  \ref{main} in the particular case when we start from the connection on $S$ that is induced by the Levi-Civita connection of the quaternion-K\"ahler metric, we immediately see that the structures in this case are the same.  In \cite{Bo} it is also shown that this complex structure is also the same as the complex structure obtained by Battaglia \cite{Bat} and studied by Hitchin \cite{Hit3} (and given by the zero set of the moment map on the twistor space). In Proposition 2 of \cite{Bo} it is shown that the holomorphic contact distribution $D^g$ coming from the quaternion-K\"ahler connection is tangent to $D^+$ and $D^-$ along the zero sections. Therefore, along the zero sections of $(V^+)^*$ and $(V^-)^*$ the distributions $D^g$ and the one given by the $U^*(2n)$ connection coincide. We will show that they also coincide along the real twistor lines corresponding to the points of $S\subset M$.  
   
   In Proposition 2 of \cite{Bo} is shown that for any $\overline{x}\in\overline{S}$, the vector subspace $D_{\overline{x}}^+$ (which is the intersection of the hypersurface $D^+$ with the fibre over $\overline{x}$ of $(V^+)^*$) is tangent to  $D^g$ and the analogous fact is true for $D^-$.

   Fix a canonical real twistor line $\sigma_x$, i.e., the twistor line that is an image in $Z$ of a fibre of the bundle $\mathcal{P}$ (see Section \ref{p}) over a point lying in the real submanifold $S_{\mathbb{R}}\subset S\times \overline{S}$. In $(V^+)^*$ (resp. in $(V^-)^*$) it is the line given explicitly by $\{(\overline{x},f\mapsto \lambda f(x))\}\subset (V^+)_{\overline{x}}$ for  $\lambda \in \mathbb{C}$ (resp. $\{(x,\tilde{f}\mapsto \tilde{\lambda} \tilde{f}(\overline{x}))\}\subset (V^+)_{\overline{x}}$). The $\mathbb{C}^*$-action preserves the twistor line (by explicit check or by equivalently the fact that this twistor line corresponds to a fixed point of the underlying $S^1$-action on $M$). Moreover, after removing two fixed points lying on the zero sections of  $(V^+)^*$ and $(V^-)^*$, the action acts transitively with the limits as $\lambda\in \mathbb{C}^*$ tends to $0$ or $\infty$ being the two fixed points.

   As both the Levi-Civita connection and the $U^*(2n)$ connection are $S^1$-invariant, the corresponding distributions on the twistor space are $\mathbb{C}^*$-invariant.  As the action is the scalar multiplication in the fibres of $(V^+)$ and $(V^-)$, we get that the affine subspaces of $(V^+)_{\overline{x}}$ modeled on $D_{\overline{x}}^+$ are also tangent to both distributions along $\sigma_x\cap (V^+)^*$ and the same is true for affine subspaces modeled on $D_x^-$ along $\sigma_x\cap (V^-)^*$. 
   As the affine subspaces of $(V^+)^*$ and of $ (V^-)^*$ are transversal, they generate the rank $n$-distribution along $\sigma_x\cap (V^+)^*\cap (V^-)^*$ which is tangent, and hence equal to the two connection distributions. We have already shown that the two connection distribution coincide also at points on the zero sections we conclude that indeed 
they coincide along the real twistor lines corresponding to the points of $S\subset M$.  

    Now we can apply the Proposition 5.3 from \cite{Alex} to conclude that the two connections agree on $S$.  
    
\end{proof}
\begin{rem}
It would be nice to prove an analogous result for any connection, namely that the  $GL(n,\mathbb{H})U(1)$ connection obtained from $D\in[D]_c$ restricts to $D$
 on $S$. However, we cannot apply methods from Theorem \ref{wn} as, except from quaternion-K\"ahler case,  we do not have explicit connection distributions on twistor space. 
  
\end{rem}
Using result of Hitchin \cite{Hit3} we have shown that in a particular case when $D$ is K\"ahler and the obtained manifold is quaternion-K\"ahler (compatible with $D$), the obtained $GL(n,\mathbb{H})U(1)$ connections have holonomy which is in fact contained in $U^*(2n)$. Motivated by this we would like to conclude with the following question:
\begin{que}
What are sufficient and necessary conditions on $(S,J,[D]_c,\mathfrak{L},\nabla, D) $ that will imply that the connection preserving the complex structure constructed in Proposition \ref{main} has holonomy contained in $U^*(2n)$?
\end{que}

\end{document}